\documentclass[12pt,reqno]{amsart}
\usepackage{mathrsfs}
\usepackage[T1]{fontenc}
\usepackage{amsmath}
\usepackage{amsthm}
\usepackage{amsfonts}
\usepackage{amssymb, amsmath}
\newtheorem{prop}{Proposition}
\newtheorem{cor}{Corollary}
\theoremstyle{definition}
\newtheorem{defn}{Definition}
\theoremstyle{plain}
\newtheorem{lmm}{Lemma}
\newtheorem{thm}{Theorem}

\newcommand{\R}{\mathbb{R}} 
\newcommand{\Z}{\mathbb{Z}}

\newcommand{\1}{\mathbf{1}} 
\newcommand{\ud}{\mathrm{d}} 
\theoremstyle{remark}

\begin{document}

\title[]
{Bounds on Sobolev norms for the nonlinear Schr\"{o}dinger equation on general tori}

\author{F. Catoire}
\address{Departement de Mathematique, Universite Paris Sud, 91405 Orsay Cedex, FRANCE }
\email {fabrice.catoire@math.u-psud.fr}

\author{W.-M. Wang}
\address{Departement de Mathematique, Universite Paris Sud, 91405 Orsay Cedex, FRANCE} 
\email{wei-min.wang@math.u-psud.fr}

\thanks{We thank P. G\'erard and N. Burq for informative conversations. W.-M. Wang also thanks J. Bourgain and L. Clozel for helpful discussions.}

\date{}



\begin{abstract} We prove Strichartz estimates on general flat $d$-torus for arbitrary $d$.
Using these estimates, we prove local wellposedness for the cubic nonlinear Schr\"odinger 
equations in appropriate Sobolev spaces. In dimensions $2$ and $3$, we prove polynomial
bounds on the possible growth of Sobolev norms of smooth solutions.  
\end{abstract}

\maketitle

\section {Introduction}

We consider the cubic nonlinear Schr\"odinger equation (NLSE) on the general flat $d$-torus $\Bbb T^d$:
\begin{eqnarray}
i \partial_t u + \Delta u = |u|^2u,\nonumber\\
u(0,x) = u_0(x),\label{Samba}
\end{eqnarray}
where $$\Bbb T^d:=\Bbb R^d/\prod_{j=1}^d\alpha_j\Bbb Z,\, 1/2\leq\alpha_j\leq 2\,, j=1,...,d,$$
and $u$ the solution to the Cauchy problem (\ref{Samba}) is a complex valued function on $\Bbb R\times\Bbb T^d$. The $L^2$ norm of $u$ is conserved. If $u_0\in H^1$, then $\|u\|_{H_1}$ is uniformly bounded in time.  

We say that the Cauchy problem is {\it uniformly (smoothly)
wellposed} in $H^s$ when, for any $R>0$, there exists $T>0$ and a Banach
space $X_T$ embedded in $\mathcal{C}^0([0,T],H^s)$ such that :

$\bullet$ for any $u_0 \in H^s$ such that $\|u_0\|_{H^s} \leq R$, (\ref{Samba}) has a
unique solution $u \in H^s$ 

$\bullet$  if $u_0 \in H^{s'}$ with $s'>s$ then $u \in \mathcal{C}^0([0,T],H^{s'})$

$\bullet$  the map $u_0 \mapsto u$ is uniformly continuous ($\mathcal{C}^\infty$).

\smallskip
In this paper, we prove the following two results.

\begin{thm}
Let $\Bbb T^d$ be a general $d$-torus. Then the cubic Schr\"odinger equation (1) is locally well posed in 
$H^s$ for any $s$ satisfying 
\begin{eqnarray*}
s&>&\frac{1}{3},\qquad\qquad\qquad d =2, \\
    &>&\frac{d}{2}-\frac{d}{d+1},\quad\quad\, \textrm{ $d\geq3$, d odd,} \\
    &>&\frac{d}{2}-1, \qquad\qquad\,\,\textrm{$d\geq4$, d even.} 
\end{eqnarray*}
\end{thm}

\begin{thm}
For $d\leq 3$, the cubic Schr\"odinger equation (1) is globally well posed in $H^s$ for 
$s\geq 1$. Moreover if $u_0\in H^s$ ($s\geq 1$), then 
$$
\|u(t)\|_{H^s} \lesssim t^{A(s-1)},
$$
for every $A$ such that 
\begin{eqnarray*}
A&>&\frac{3}{2},\quad d=2, \\
       &>& \frac{15}{2}, \,\,\,d=3. \\
\end{eqnarray*}
\end{thm}

The main ingredient for the proofs of Theorems 1 and 2 is an $L^4$ Strichartz estimates on 
the Schr\"odinger semi-group (Proposition 1). By Fourier series, this is reduced to asymptotic
estimates on the number of integer points in an elliptic annulus (Lemma 1). This line of
research was initiated by Bourgain in the series of papers [Bou1-4]. Here we work out
the case of general d-torus for arbitrary $d$, where no number theoretical assumptions 
are made on the ratios of the $\alpha_j$.  

When $d=2$, we use the geometric argument of Janick \cite {J} to obtain the exponent $s_0=1/3$
in Theorem 1 for local wellposedness. We note that the generic estimate for compact $2$-manifolds
is $s_0=1/2$ from \cite{BGT1}. Here it is important to remark that the improvement comes 
from the strict convexity of the level sets (elliptic annulus). Otherwise the length of the boundary
over the real and the integers are of the same order giving $s_0=1/2$. For example, for the 
circular disc of area $\pi R$, $R^{s_0}$ is precisely the error term in counting the number
of integers in the disc.  Viewing the problem 
this way, a natural lower bound for $s_0$ would be $1/4$ from \cite{CdV}, while the best 
upper bound obtained so far is $7/22$ \cite{Bom}.

When $d\geq 3$, our argument mainly uses analysis. (The argument also applies to $d=2$,
but it only gives $s_0=1/2$.) When $d=3$, we obtain $s_0=3/4$. In \cite{Bou4}, using a more 
involved analysis, this is improved to $s_0=2/3$. 

Once we have established the Strichartz estimates, the proof proceeds via the routine of $X^{s,b}$ spaces:
\begin{defn}
$$
X^{s,b}:=\{u \in \mathcal{S}' \textrm{ so that } \|u\|_{X^{s,b}} < \infty\}
$$
where
$$
\|u\|_{X^{s,b}} := \|e^{-it\Delta}u(t,.)\|_{H^b(H^s)} = \| (1+|i\partial_t +
\Delta|²)^{b/2}(1-\Delta)^{s/2} (u)\|_{L^2}.
$$
We denote $X_T^{s,b}$ the set of restrictions of these functions to $[0,T]$ 
with its norm : 
$$
\|u\|_{X^{s,b}}:=\inf(\|v\|_{X^{s,b}} \textrm{ , } v \textrm{ such that } v_{|[0,T]}=u).
$$
\end{defn}

The organization of the paper is as follows: in section 2, we prove the Strichartz estimates, in section 3, local well-posedness and finally in section 4, we prove bounds on Sobolev norms in dimensions 2 and 3 for smooth initial data.  

\section{Strichartz estimates for general tori}
\begin{prop} 
If $f \in L^2(\Bbb T^d)$ whose spectrum lies in $[-N,N]^d$, then 
\begin{eqnarray*}
\| e^{it \Delta} f\|²_{L^4_t L^4_x} &\lesssim&N^{\frac{1}{3}}\|f\|^2_{L^2},\qquad\qquad\qquad  \textrm{ $d =2$,} \\
    &\lesssim&N^{\frac{d}{2}-\frac{d}{d+1}+\epsilon}\|f\|^2_{L^2},\quad\quad\quad\, \textrm{  $d\geq3$, d odd,} \\
    &\lesssim&N^{\frac{d}{2}-1+\epsilon}\|f\|^2_{L^2}, \qquad\qquad\,\,\textrm{    $d\geq4$, d even,} 
\end{eqnarray*}
where $L^4_t $ denotes $L^4_t$(loc) and  $\epsilon>0$ is arbitrarily small.
\end{prop}
 
The proof of the proposition reduces to estimates on the (near) degeneracy factor. Precisely,
\begin{lmm} 
If $f \in L^2(\Bbb T^d)$ whose spectrum lies in $[-N,N]^d$, then 
$$
\| e^{it \Delta} f\|²_{L^4_t L^4_x} \lesssim 
\| \# A_\ell\|_\infty^{1/2}\|f\|^2_{L^2},
$$
where $A_\ell=\{m\in [-N,N]^d \,|\,|Q(m)-\ell|\leq 1\}$ and $\# A_\ell$ is the cardinal number of 
$A_\ell$.
\end{lmm}

We first prove a simple lemma on almost orthogonality.
\begin{lmm}
Let $H$ be a Hilbert space and $(H_k)_{k \in [-N,N]^d}$ subspaces of $H$.
Suppose that for any $l$ and $m$, if $h_\ell \in H_\ell$ and $h_m \in H_m$
then 
$$
\langle h_\ell| h_m \rangle \leq \frac{\|h_\ell\| \|h_m\|}{1+|\ell-m|^2}. 
$$
Then for any sequence $(h_k) \in \prod_k H_k$,
$$
\| \sum_k h_k \|² \leq C \sum_k \|h_k\|² .
$$  
\end{lmm}
\begin{proof} 
let $e_k= h_k/\|h_k\|$ and 

\begin{eqnarray*}
\phi :            \ell^2                &\to&        H, \\
                 (c_k)            &\mapsto& \sum_k c_k e_k. \\
\end{eqnarray*}
Then, $||| \phi |||^2²_{\ell^2\to H} = ||| {}^t\phi \phi |||_{\ell^2\to\ell^2}$.
Recall that $\psi={}^t\phi \phi$ is associated to the matrix $(\langle e_\ell| e_m \rangle )_{\ell,m}$.
We have assumed that 
$$
\max_\ell \sum_m |\psi_{\ell,m}| \leq C < \infty
$$
and
$$
\max_m \sum_\ell|\psi_{\ell,m}|\leq C < \infty.
$$
Hence by Schur's lemma :
$$
||| {}^t\phi \phi |||_{\ell^2 \to \ell^2} \leq C,
$$  
where $C$ is a constant that does not depend on the choice of $h_k$, so the proof is complete.
\end{proof}

We now prove Lemma 1:
\begin{proof}
\begin{eqnarray} 
\| e^{it \Delta} f\|²_{L^4_t L^4_x} 
&=& \| (e^{it \Delta} f)²\|_{L^{p/2}_t L^2_x}\nonumber\\
&=& \| (\sum_{a \in \Z^d} | \sum_n \hat{f}(n) \hat{f}(a-n)
e^{it(Q(n)+Q(a-n))}|²)^{1/2} \|_{L^{2}_t}\nonumber\\
&\leq& ( \sum_{a\in\Z^d} \| \sum_n  \hat{f}(n) \hat{f}(a-n)
e^{it(Q(n)+Q(a-n))} \|²_{L^{2}_t})^{1/2}.\label {star}\\ 
\nonumber\end{eqnarray}
In order to compute efficiently, we shall decompose: 
\begin{eqnarray*} 
g_a(t)
&=& \sum_n \hat{f}(n) \hat{f}(a-n) e^{it(Q(n)+Q(a-n))} \\ 
&=& \sum_k \sum_{|Q(n)+Q(a-n) -k| \leq 1/2}  \hat{f}(n) \hat{f}(a-n)
e^{it(Q(n)+Q(a-n))}.\\  
\end{eqnarray*}
Then according to Lemma 2 and let $c_n=|\hat{f}(n)|$,
one has :
$$
\|g_a\|_{L^2}
\lesssim \| \sum_{|Q(n)+Q(a-n) -k| \leq 1/2}  c_n c_{a-n}\|²_{\ell^{2}_k}
$$
Let $Q(n)+Q(a-n)= 2(Q(2n-a)+Q(a))$, $\ell=2k-Q(a)$
and  $A_\ell = \{m \in [-N,N]^d \textrm{ such that }
|Q(m) -\ell| \leq 1 \}$,
the condition $|Q(n)+Q(a-n) -k|\leq 1/2$ becomes $2n \in A_{\ell}$.
Then, 
\begin{eqnarray*}
\|g_a\|_{L^2}
&\lesssim& (\sum_\ell(\sum_{2n \in A_\ell} c_n c_{a-n})^2)^{1/2} \\
&\lesssim& (\sum_\ell \# A_\ell (\sum_{2n \in A_l} c_n² c_{a-n}²))^{1/2} \\
&\lesssim& (\sup_\ell \# A_\ell)^{1/2} (\sum_n c_n² c_{a-n}²)^{1/2} \\
\end{eqnarray*}
using Cauchy-Schwarz.
Inserting this in (\ref{star}), we obtain Lemma 1.
\end{proof}
\smallskip

\noindent{ \it (i) Estimates on $ \|\# A_\ell\|_\infty $ for} $d\geq 3$.

\begin{lmm} 
$$
\| \# A_\ell \|_{\infty} \lesssim N^{2s_0},
$$
where \begin{eqnarray*}
s_0
&=& \frac{d}{2}-\frac{d}{d+1}+\epsilon,\quad \textrm{ $d\geq3$, d odd,} \\
&=&  \frac{d}{2}-1+\epsilon,\quad\quad\quad\textrm{ $d\geq4$, d even,} \\
\end{eqnarray*}
and $\epsilon>0$ is arbitrarily small.
\end{lmm}

\begin{proof}
Let $\phi$ be a function of fast decay at infinity so that $\hat{\phi}(\tau)\geq 0$ and 
$\hat{\phi}(\tau)\geq1$ on $[-1,1]$. We have therefore 
\begin{eqnarray*}
\# A_\ell &\lesssim& \sum_m \hat{\phi}(Q(m)-\ell)\\
&=&\int \sum_m e^{itQ(m)}e^{-i\ell t} \phi(t)dt. 
\end{eqnarray*}
Hence 
$$
\| \# A_\ell \|_{\ell^{\frac{p}{p-4}}_\ell}
\lesssim (\int |\sum_m e^{itQ(m)}|^{p/4} |\phi(t)| dt )^{4/p}\quad(4\leq p\leq 8),
$$
which is obvious when $p=4$ and can be shown by almost-orthogonality when
$p=8$. By interpolation, the previous inequality holds for any $4  \leq p
\leq 8$. 

Then, 
\begin{eqnarray*}
\| \# A_\ell \|_{\ell^{\frac{p}{p-4}}_\ell}
&\lesssim& (\int |\sum_m e^{it Q(m)} |^{p/4} |\phi(t)| dt)^{4/p} \\
&\lesssim& (\int | \sum_m e^{it Q(m)} |^{p/4}  |\phi(t)| dt)^{4/p} \\
&\lesssim& (\int | \sum_{m_1} e^{it \theta_1 m_1²} |^{p/4} 
                 ... 
                 | \sum_{m_d} e^{it \theta_2 m_3²} |^{p/4}
                 |\phi(t)| dt)^{4/p} \\
&\lesssim& (\int | \sum_{-N \leq k \leq N} e^{itk²} |^{dp/4} {} |\phi(t)| dt
)^{4/p} \quad(4\leq p\leq 8). \\ 
\end{eqnarray*}

Suppose $d$ is even and  write it as $2d'$. Let $p=4$, we have
\begin{eqnarray*}
\| \# A_\ell \|_{\ell^{\infty}}
&\lesssim& \int | \sum_{-N \leq k \leq N} e^{itk²} |^{2d'} {} |\phi(t)| dt\\ 
&\lesssim& \int |\sum_{-N \leq k_1 , ..., k_d\leq N} e^{it(k_1²+...+k_{d'}^2)}|^2 
|\phi(t)| dt  \\ 
&\lesssim& \int |\sum_{m \leq N^2} r_{d'}(m) e^{itm}|² |\phi(t)| dt,
\end{eqnarray*}
where $r_{d'}(m) = \#{(k_1,...,k_{d'}) \in \Z^{d'} \textrm{ so that }
  m=k_1²+...+k_{d'}^2}$.  
Since $r_{d'}(m) \lesssim m^{(d-2)/2 + \varepsilon}$ for $d\geq 4$, we obtain
the lemma for even $d$.  

For odd $d$, $d\geq 3$, take $p=4\frac{d+1}{d}$ and write $d+1=2d'$. We have 
\begin{eqnarray*}
\| \# A_\ell \|_{\ell^{\infty}}&\leq &\| \# A_\ell \|_{\ell^{d+1}}\\
&\lesssim& (\int | \sum_{-N \leq k \leq N} e^{itk²} |^{d+1} {} |\phi(t)| dt)^{\frac{d}{d+1}}\\ 
&\lesssim& (\int |\sum_{-N \leq k_1 , ..., k_d\leq N} e^{it(k_1²+...+k_{d'}^2)}|^2 
|\phi(t)| dt )^{\frac{d}{d+1}} \\ 
&\lesssim& (\int |\sum_{m \leq N^2} r_{d'}(m) e^{itm}|² |\phi(t)| dt )^{\frac{d}{d+1}},
\end{eqnarray*}
which gives the lemma for odd $d$.
\end{proof}

The argument above gives $s_0=1/2$ for $d=2$, which is the generic bound for compact 
2-manifolds proven in \cite{BGT1}. In the present case, using convexity and Janick's \cite{J} geometric
proof, we improve the bound to $s_0=1/3$. The Janick argument works in arbitrary $d$. 
For us, it is only useful for $d=2$. 

\noindent {\it (ii) Estimates on $ \|\# A_\ell\|_{\infty}$ for} $d=2$.

\begin{lmm} \label{César}
Assume $\Sigma_1$ is a  closed stritctly convex hyper-surface in $\R^d$
containing $0$ in its convex envelope so the curvature is strictly positive. 
Suppose for all $X \in \R^*_+$, $\Sigma_X$ is the image of $\Sigma_1$ by
homothety of center $0$ and scale $X^{1/2}$. For any $d+1$ non coplanar 
integer points in the annulus formed by  $\Sigma_X$ and $\Sigma_{X+1}$, the
largest pairwise distance is at least $CX^{\frac{1}{2(d+1)}}$,
where $C$ only depends on the curvature of $\Sigma_1$ and $d$.
\end{lmm}

\begin{cor} For $d=2$
$$
\| \# A_\ell \|_{\infty} \lesssim N^{2/3}.
$$
\end{cor}

\begin{proof} For $d=2$, Lemma 4 gives that the largest distance among $3$ non colinear 
points $\gtrsim N^{1/3}$. Since the number of colinear points in the elliptic annulus is 
finite (uniform in $N$), this proves the corollary.
\end{proof}

\noindent{\it Proof of Proposition 1.} Inserting Lemma 3 and Corollary 1 into Lemma 1 gives 
Proposition 1.
\hfill $\square$

We now prove Lemma 4.
\begin{proof}
We first suppose that these $d+1$ non coplanar points: $A_1$, ...,  $A_{d+1}$
are all on the hyper-surface $\Sigma_X$. We can always assume that the largest distance 
between pars of points is less than $X^{1/2}$.
Since $\Sigma_X$ is strictly convex, the polyhedron $A_1
... A_{d+1}$ is not flat and its volume $1/d *
|\det(\overrightarrow{A_1 A_2},... , \overrightarrow{A_1 A_{d+1}})|$
is at least $1/d$.  Let  $A'_1$, ... , $A'_{d+1}$ 
be the homothetics respectively of $A_1$, ... , $A_{d+1}$ according to the scale
$X^{-1/2}$ and $D$ the largest pairwise distance of these points. 
The  volume of the  polyhedron  $A'_1...A'_{d+1}$ is of order $X^{-d/2}/d$.
($D<1$, in view of the restriction on the pairwise distance of points of $A_1,...,A_{d+1}$.) 

Suppose we are in a coordinate system such that $A'_1$ maximises the abscissa, then
the difference in the ordinates is less than $D$ and the absisse $D^2$. 
So
$$
X^{-d/2} /d  \lesssim D^{d+1}
$$
Hence $D \gtrsim X^{-d/(2(d+1))}$ and the largest pairwise distance of 
$A_j$ is of order $X^{1/2} * X^{-d/(2(d+1))} =
X^{1/(2(d+1))}$.

We now only assume that the $A_j$ are non coplanar (and that the largest pairwise distance 
is less than $X^{1/2}$).
The volume of the  polyhedron 
$A_1...A_{d+1}$ is  
again at least $1/d$. Project the $d+1$ points $A_j$ onto 
$\Gamma_X$ and name them respectively as $A^\#_j$ (i.e. ${A^\#_j} = [O, A_j] \cap
\Sigma_X$). We remark that $A_jA^\#_j \lesssim X^{-1/2}$, which can be seen as follows. 
Let $O$ be the origin, then $OA_j = OA^\#_j + A^\#_jAj$, $OA^\#_j = OA'_j *X^{1/2}$ and
$OA_j \leq OA'_J * (X+1)^{1/2}$. So $A_jA^\#_j \leq OA_1 ((X+1)^{1/2}
- X^{1/2})$, which gives $A_jA^\#_j \leq r X^{-1/2}$, where $r$ is the distance between the origin and 
$\Gamma_1$. 

So the volume of the polyhedron $A^\#_1...A^\#_{d+1}$ is at least $1/2d$.
As before, let $D$ be the 
largest  pairwise distance of the homothetics $A'_j$ of $A^\#_j$ : 

\begin{eqnarray*}
\det(\overrightarrow{A_1A_2},..., \overrightarrow{A_1A_{d+1}})
- \det(\overrightarrow{A^\#_1A^\#_2},..., \overrightarrow{A^\#_1A^\#_{d+1}})
= \\
\det(\overrightarrow{A_1A_2} , \overrightarrow{A_{d+1}A^\#_{d+1}}) 
- \det(\overrightarrow{A_1A_2} , \overrightarrow{A_1A^\#_1}) \\
+... \\
+\det(\overrightarrow{A_2A^\#_2} , \overrightarrow{A^\#_1A^\#_{d+1}})
-\det(\overrightarrow{A_1A^\#_1} , \overrightarrow{A^\#_1A^\#_{d+1}}) \\
\leq d D * X^{-1/2},
\end{eqnarray*}
which shows that the difference in volume is $o(1)$ since $D<X^{1/2}$. This concludes
the proof using the previous argument on the hyper-surface.
\end{proof} 
\smallskip

\section{Local wellposedness}
\begin{prop}
If $f_1$ and $f_2 \in L^2(\Bbb T^d)$ whose spectra lie in $[-N_1, N_1]^d$ and
$[-N_2, N_2]^d$, then
\begin{equation}
\| e^{it \Delta} f_1  e^{it \Delta} f_2  \|_{L^2_t L^2_x}
\lesssim \min(N_1,N_2)^{s_0} \|f_1\|_{L^2} \|f_2\|_{L^2}, \label {*}
\end{equation}
where 
\begin{eqnarray*}
s_0&=&\frac{1}{3},\qquad\qquad\qquad\quad d=2,\\
&=& \frac{d}{2}-\frac{d}{d+1}+\epsilon,\quad \textrm{ $d\geq3$, d odd,} \\
&=&  \frac{d}{2}-1+\epsilon,\quad\quad\quad\textrm{ $d\geq4$, d even,} \\
\end{eqnarray*}
and $\epsilon>0$ is arbitrarily small.
\end{prop}
\begin{proof}
Suppose $N_1 \leq N_2$. It is easy to show that Proposition 1 holds more generally 
for $f\in L^2(\Bbb T^d)$ whose spectrum lies in $a+[-N, N]^d$, $a\in\Bbb Z^d$ and arbitrary. 
Let us decompose $f_2$ :
$$
f_2 = \sum_i f_2^{(i)} = \sum_i \1_{iN_1 \leq (-\Delta)^{1/2} \leq (i+1)N_1} f_2.
$$
Then using almost orthogonality and H\"older's inequality, we have

\begin{eqnarray*} 
\| e^{it \Delta} f_1  e^{it \Delta} f_2  \|_{L^2_t L^2_x} 
&=&  \| \sum_i e^{it \Delta} f_1  e^{it \Delta} f_2^{(i)}  \|_{L^2_t L^2_x} \\
&\lesssim& (\sum_i \| e^{it \Delta} f_1  e^{it \Delta} f_2^{(i)}
\|²_{L^2_t L^2_x})^{1/2} \\ 
&\lesssim& (\sum_i \| e^{it \Delta} f_1\|²_{L^4_t L^4_x} \| e^{it \Delta}
f_2\|²_{L^4_t L^4_x})^{1/2} \\
&\lesssim& N_1^{s_0} \|f_1\|²_{L^2} \sum_i  \| f_2^{(i)}\|_{L^2})^{1/2} \\
&\lesssim& N_1^{s_0} \|f_1\|_{L^2} \|f_2\|_{L^2} \\
\end{eqnarray*}
\end{proof}

\begin{prop}
Let $(M, g)$ be  a compact Riemannian $d-$manifold. Assume (\ref{*}) holds with $0\leq s_0<1$
for any $f_1$, $f_2\in L^2(M)$ with spectra in $[-N_1, N_1]^d$,
$[-N_2, N_2]^d$, then the NLSE is smoothly locally well posed in $H^s$ for every $s>s_0$.
\end{prop}
\smallskip
\noindent {\it Proof of Theorem 1.} This  follows directly from Propositions 2 and 3. \hfill $\square$

The proof of Proposition 3 uses the $X^{s,b}$ spaces, see \cite{Bou3, BGT2, Z}. For completeness, we 
reproduce the arguments. Below we assume the hypothesis in Proposition 3 holds.

\begin{prop}
If $f_1,f_2, f_3, f_4 \in L^2(M)$ whose spectra lie respectively in $[N_1, 2N_1]^d$, $[N_2, 2N_2]^d$, 
$[N_3, 2N_3]^d$, $[N_4, 2N_4]^d$ and if $\chi$ is a function compactly supported in $\R$
then, 
\begin{equation}
\sup_{\tau \in \R} \int_\R \int_M\chi(t)e^{it\tau}u_1 \overline{u_2} u_3
\overline{u_4} \ud x \ud t
\lesssim m(N_1,N_2,N_3,N_4)^{s_0}
\|f_1\|_{L^2} \|f_2\|_{L^2} \|f_3\|_{L^2} \|f_4\|_{L^2}, \label{**}
\end{equation}
where $u_j=e^{it\Delta} f_j$, $m(N_1,N_2,N_3,N_4)$ is the product of the two smallest $N_j$, $j=1,...4$
and $s_0$ as in Proposition 3.
\end{prop}

\begin{proof}
Suppose that $m(N_1,N_2,N_3,N_4)= N_1 N_3$.
Then, 
\begin{eqnarray*}
\sup_{\tau \in \R} \int_\R \int_M \chi(t)e^{it\tau} u_1 \overline{u_2} u_3
\overline{u_4} \ud x \ud t 
&\lesssim&  \| u_1 \overline{u_2} \|_{L^2} \| u_3 \overline{u_4} \|_{L^2} \\
&\lesssim& N_1^{s_0} \|f_1\|_{L^2} \|f_2\|_{L^2} *
N_3^{s_0} \|f_3\|_{L^2} \|f_4\|_{L^2} \\
&\lesssim& m(N_1, N_2,N_3,N_4)^{s_0} \|f_1\|_{L^2} \|f_2\|_{L^2}
\|f_3\|_{L^2} \|f_4\|_{L^2} \\
\end{eqnarray*}
\end{proof}

\begin{lmm}
Under the assumption (\ref{**}), 
for every $b>1/2$ and every $u_1$, $u_2$, $u_3$, $u_4$ $\in X^{0,b}$
whose spectra (relative to the space variable) lie respectively in $[N_1,2N_1]^d$, $[N_2,2 N_2]^d$, 
$[N_3, 2 N_3]^d$, $[N_4,2N_4]^d$, the following holds:
\begin{equation}
\int_\R \int_M u_1 \overline{u_2} u_3 \overline{u_4} \ud x \ud t \lesssim
m(N_1,N_2,N_3,N_4)^{s_0} \prod_{1 \le i \le 4} \|u_i\|_{X^{0,b}}. \label{***}
\end{equation}
\end{lmm}

\begin{proof}
Assume first that $u_3$ and $u_4$ are supported in $[0,1]$ in the time variable and that 
$\chi=1$ in $[0,1]$.
Let $u_j^* = e^{-it\Delta} \hat{u_j}$ and take the Fourier transformation in time, one has:
\begin{eqnarray*}
\big(u_1 \overline{u_2}u_3 \overline{u_4} \big)(t) & = & \int_\R \int_\R \int_\R
\int_\R e^{it(\tau_1 -\tau_2 + \tau_3 - \tau_4)} 
e^{it\Delta}  u^*_1(\tau_1) \overline{ e^{it\Delta} u^*_2(\tau_2)}  \\
& & \hspace{2.5cm}e^{it\Delta} u^*_3(\tau_3) \overline{e^{it\Delta} u^*_4(\tau_4)} \frac{\ud \tau}{(2 \pi)^4} 
\end{eqnarray*}
By Funini :
\begin{eqnarray*}
I &:=& \int_\R \int_M u_1 \overline{u_2} u_3 \overline{u_4} \ud x \ud t \\ 
&\ =& \int ... \int \chi(t) e^{it(\tau_1
-\tau_2 + \tau_3 -\tau_4)} e^{it\Delta} u^*_1(\tau_1) \overline{ e^{it\Delta}u^*_2(\tau_2)})
e^{it\Delta} u^*_3(\tau_3) \\
&   & \hspace{2.5cm} \overline{e^{it\Delta} u^*_4(\tau_4)} \ud t \ud x  \frac{\ud \tau}{(2\pi)^4}.
\end{eqnarray*} 

The assumption (\ref{**}) says that :
$$
|I| \lesssim m(N_1,N_2,N_3,N_4)^{s_0} \int_\R \int_\R \int_\R \int_\R
\prod_{j} \|u_j\|_{L^2(M)}(\tau_j) \ud \tau.$$
Since $(1+ \tau^2)^b$ is integrable, using Cauchy-Schwartz inequality one
obtains (\ref{***}). 

In the general case, decomposing the support of  $u_3$ and $u_4$ and using
almost orthogonality, one proves (\ref{***}).
\end{proof}

\begin{lmm} \label{azerty}
$$
\int_\R \int_M u_1 \overline{u_2}  u_3 \overline{u_4} \ud x \ud t \lesssim 
m(N_1,...,N_4)^{d/2}
\prod_{1 \le i \le 4}\|u_i\|_{X^{0,1/4}}
%\|u_1\|_{X^{0,1/4}}  \|u_2\|_{X^{0,1/4}} \|u_3\|_{X^{0,1/4}} \|u_4\|_{X^{0,1/4}}
$$
\end{lmm}

\begin{proof}
Suppose that $m(N_1,...,N_4)= N_1 N_3$, from Hölder and Sobolev inequalities:

\begin{eqnarray*}  
\int_\R \int_M u_1 \overline{u_2}  u_3 \overline{u_4} \ud x \ud t
&\lesssim& \|u_1 \overline{u_2}\|_{L^2 L^2}  \|u_3 \overline{u_4}\|_{L^2 L^2} \\
&\leq& C \|u_1\|_{L^4 L^\infty} \|u_2\|_{L^4 L^2} \|u_3\|_{L^4 L^\infty}
\|u_4\|_{L^4 L^2} \\
&\leq& C (N_1 N_3)^{d/2} \|u_1\|_{L^4 L^2} \|u_2\|_{L^4 L^2}
\|u_3\|_{L^4 L^2} \|u_4\|_{L^4 L^2} \nonumber \\
&\leq& C (N_1 N_3)^{d/2} \prod_{1 \le i \le 4} \|u_i\|_{X^{0,1/4}},
\end{eqnarray*}
using the embedding $X^{0,1/4} \subset L^4L^2$ (Sobolev injection in $t$ applied to 
$e^{-it\Delta}u(t)$).
\end{proof}

\begin{lmm} \label{L}
For every $s>s_0$, there exists $b<1/2$ such that (\ref{***}) holds (with $s$ replacing $s_0$).
\end{lmm}

\begin{proof}
Decompose $u_j$ as follows:
\begin{eqnarray*}
u_j &= &\sum_{K_j} u_{j,K_j} \\
u_{j,k_j} & = &\mathbf{1}_{K_j\leq 1+
  |i\partial_t + \Delta| \leq  K_{j+1}} (u_j)
\end{eqnarray*}
where $K_j$ are dyadic integers.
Hence,
$$
\|u_j\|²_{X^{0,b}} \simeq \sum_{K_j} K_j^{2b} \|u_{j,K_j}\|²_{L^2} \simeq
\sum_{K_j} \|u_{j,K_j}\|_{X^{0,b}}
$$
Then, if $b>1/2$, the previous two lemme can be interpreted as
\begin{eqnarray*}
|I(u_1,u_2,u_3,u_4)| & \leq &
  C m(N_1,N_2,N_3,N_4)^{s_0} \\
& & \sum_{K_1,K_2,K_3,K_4} (K_1K_2K_3K_4)^{b} \prod_j \|u_{j,K_j}\|²_{L^2}
\end{eqnarray*}
and
\begin{eqnarray*}
|I(u_1,u_2,u_3,u_4)| & \leq & C m(N_1,N_2,N_3,N_4)^{d/2} \\
& & \sum_{K_1,K_2,K_3,K_4}
(K_1K_2K_3K_4)^{1/4} \prod_j \|u_{j,k_j}\|²_{L^2}
\end{eqnarray*}
Hence for  $s>s_0$, one can choose $b$ sufficiently close to $1/2$ and interpolate between the two inequalities to obtain :
\begin{eqnarray*}
|I(u_1,u_2,u_3,u_4)| & \leq & C m(N_1,N_2,N_3,N_4)^{s} \\
& & \sum_{K_1,K_2,K_3,K_4}
(K_1K_2K_3K_4)^{b} \prod_j \|u_{j,k_j}\|²_{L^2} 
\end{eqnarray*}
with $b<1/2$.
\end{proof}

\begin{lmm} \label{Gabriel}
If $b$ and $b'$ are such that $0 \leq b' < 1/2$ and $0 \leq b+b' < 1$ and if
$T \in [0,1]$,  then there exists C such that if $w(t)=\int_0^t S(t-t') f(t')
dt' $
then $\|w\|_{X_T^{s,b}} \leq C T^{1-b-b'} \|f\|_{X_T^{s,-b'}}$.
\end{lmm}

\begin{proof}
See \cite{G} for a proof of this lemma
\end{proof}

\begin{lmm} \label{zoinx}
If $s>s_0$ then there exist $b$ and $b'$ such that $0<b'<1/2<b$ and $b+b'<1$
so that :
$$ 
\|u_1 \overline{u_2} u_3\|_{X_T^{s,-b'}} \lesssim \|u_1\|_{X_T^{s,b}}
\|u_2\|_{X_T^{s,b}} \|u_3\|_{X_T^{s,b}} 
$$
\end{lmm} 
 
To prove Lemma {\ref {zoinx}}, we need the following inequality, (see \cite{BGT2} for a proof).
\begin{lmm} \label{zoin}
Let $P_\lambda$ ($\lambda\geq 0$) be the orthogonal projection onto $\text {ker}\, (-\Delta-\lambda^2)$. 
There exists $C>0$ such that 
if $0\leq\lambda_j\leq\lambda_4$ for $j=1,2,3$ then for every $p>0$, there exists
$C_p$ so that for every $w_j \in L^2(M)$ 
\begin{multline}
\int_M P_{\lambda_1}(w_1) P_{\lambda_2}(w_2) P_{\lambda_3}(w_3) P_{\lambda_4}(w_4) dx
\leq \\ C_p \lambda_{4}^{-p} \|w_1\|_{L^2} \|w_2\|_{L^2} \|w_3\|_{L^2} \|w_4\|_{L^2}.\nonumber
\end{multline}
\end {lmm}

\noindent{\it Proof of Lemma 9.}
By a duality argument, one only needs to prove:
$$
\bigg|\int_{\R \times M} u_1 \overline{u_2} u_3 \overline{u_4} dx \ dt \bigg|
\lesssim \|u_1\|_{X^{s,b}} \|u_2\|_{X^{s,b}} \|u_3\|_{X^{s,b}} \|u_4\|_{X^{-s,b'}}.
$$
Decomposing the four function as :
\begin{eqnarray*}
u_j & = & \sum_{N_j} u_{j, N_j} \\
u_{j, N_j} & = & \mathbf(1)_{\sqrt{1-\Delta} \in [N_j,2N_j]} (u_j)
\end{eqnarray*}
with $N_j$ being dyadic integers.
Then the integral can be writen as the sum of terms of the form: 
$$
J(N_1,N_2,N_3,N_4)=\int_{\R \times M} u_{1,N_1} \overline{u_{2,N_2}} u_{3,N_3}
\overline{u_{4,N_4}} dx \ dt .
$$
Without restrictions, suppose that $N_1 \leq N_2 \leq N_3$.
Let $s'$ be such that $s_0<s'<s$.
Lemma 7 gives $b'<1/2$ so that 
$$
|J(N_1,N_2,N_3,N_4)| \lesssim (N_1 N_2)^{s'} \|u_1\|_{X^{0,b'}}
\|u_2\|_{X^{0,b'}} \|u_3\|_{X^{0,b'}} \|u_4\|_{X^{0,b'}} 
$$
Hence,
\begin{multline*}
\bigg|\int_{\R \times M} u_1 \overline{u_2} u_3 \overline{u_4} dx \ dt \bigg|
 \lesssim \\ \sum_{N_j}  (N_1 N_2)^{(s'-s)} (\frac{N_4}{N_3})^{s}
\prod_{j=1}^3\|u_{j,N_j}\|_{X^{s,b'}} \|u_{4,N_4}\|_{X^{-s,b'}}.
\end{multline*}

We distinguish two types of terms, the ones with $N_4 \leq C
N_3$ and the others. For the first ones, we use Cauchy-Schwarz inequality to bound by :
$$
(\sum_{N_j} (N_1 N_2)^{2(s'-s)} (\frac{N_4}{N_3})^{2s} \sum_{N_j} 
\prod_{j=1}^4\|u_{j,N_j}\|²_{X^{s,b'}} \|u_{4,N_4}\|²_{X^{-s,b'}})^{1/2}. 
$$ 
The prefactor is bounded by :
$$
(\sum_{N_1} N_1^{s'-s} \sum_{N_2} N_2^{(s'-s)/2} \sum_{N_3}
(N_3^{(s'-s)/2} \sum_{N_4 \leq C N_3} (\frac{N_4}{N_3})^{2s} ) )^{1/2}.
$$
Every series is bounded, since the sum is over dyadic integers, the fourth series is equivalent to its last term $N_3^{2s} / N_3^{2s}$.
For the terms so that  $N_4 > C N_3$, we use Lemma {\ref{zoin}}.

Let $p>s$, one has:
\begin{eqnarray*}
|J(N_1,N_2,N_3,N_4)|
&\lesssim& N_4^{-p} \int_\R \prod_j \|u_{j,N_j}(t)\|_{L²} dt \nonumber \\
&\lesssim& N_4^{-p} \prod_j \|u_{j,N_j}(t)\|_{L^4 L^2} \nonumber \\
&\lesssim& N_4^{-p} \prod_j \|u_{j,N_j}(t)\|_{X^{0,1/4}} \nonumber \\
&\lesssim& N_4^{-p} \prod_j \|u_{j,N_j}(t)\|_{X^{0,b'}} \nonumber \\
% &\lesssim& N_4^{-p} \|u_{1,N_1}(t)\|_{X^{0,b'}}
% \\ & & \hspace{5mm} \|u_{j,N_j}(t)\|_{X^{0,b'}}
% \|u_{j,N_j}(t)\|_{X^{0,b'}} \|u_{j,N_j}(t)\|_{X^{0,b'}} \nonumber \\ 
&\lesssim& (N_1 N_2 N_3)^{-s} N_4^{-p+s} \|u_{1,N_1}(t)\|_{X^{s,b'}}
\\ & & \hspace{5mm} \|u_{2,N_2}(t)\|_{X^{s,b'}} \|u_{3,N_3}(t)\|_{X^{s,b'}} \|u_{4,N_4}(t)\|_{X^{-s,b'}}.
\end{eqnarray*}

Using Cauchy-Schwarz, the above is bounded by 
$$\|u_1\|_{X^{s,b'}}
\|u_2\|_{X^{s,b'}} \|u_3\|_{X^{s,b'}} \|u_4\|_{X^{-s,b'}}.$$
Choosing $b$ such that $1/2<b<1-b'$ concludes the proof.
\hfill $\square$

\noindent{\it Proof of Proposition 3.}
Lemme 8 and 9 show that the second term in the Duhamel formula:
$$u(t)=e^{it\Delta}u_0-i\int_0^te^{i(t-\tau)\Delta}(|u(\tau)|^2)u(\tau)d\tau$$
is a contraction in $X^{s, b}_T$ and for $T$ sufficiently small has norm less
than $1$. Moreover this norm only depends on $b$, $b'$ (cf. Proposition 2.11 of [BGT2]).
Hence we have proved Proposition 3. In case $s\geq 1$, $T$ can be shown \cite{Z} to
depend only on $\|u\|_{H_1(M)}$, which is uniformly bounded in time. \hfill $\square$. 
\smallskip
\section{Bounds on Sobolev norms}
Theorem 2 follows from the following:
\begin{prop} Let $(M, g)$ be a compact Riemannian $d$-manifold. Assume (\ref{*}) holds with $0\leq s_0<1$
for any $f_1$, $f_2\in L^2(M)$ with spectra in $[-N_1, N_1]^d$, $[-N_2, N_2]^d$. 
If $u\in C(\Bbb R, H^s(M))$ is the solution to the Cauchy problem with the initial datum 
$u_0\in H^s(M)$, $s\geq 1$:
\begin{eqnarray*}
i\partial_t u + \Delta u &=& |u|^2 u \\
u(0,x)&=&u_0(x),
\end{eqnarray*}
then
$$
\|u(t)\|_{H^s} \lesssim t^{A(s-1)}
$$
for every $A$ such that
\begin{eqnarray*}
A^{-1} &<& 1-s_0,\qquad\qquad\qquad\qquad\qquad\quad\, \textrm{ if $d=2$, }\\
       &<& (1-\frac{d-2}{2(d-2s_0)})(1-\frac{d-2}{d-s_0-1}), \textrm{ if $d\geq 3$ }. \\
\end{eqnarray*}
\end{prop}

For completeness, we present the proof of the proposition, which is adapted from the proof 
in \cite{Z}, see also \cite{Bou3}.

\begin{proof}
Suppose $s$ is an even integer written as $2r$. This is no restriction as once the proposition is proven for every even integer, by interpolation, the result holds for every $s>1$. From Proposition 3, the local 
existence time $T$ only depends on the $H^1$ norm. By elementary iteration process, the solution
is therefore global. 

Let $t_j= j *T/2$.
Since $L^2$ norm is conserved, one only needs to majorize $\| \Delta^r u \|_{L^2}$.
Hence :
$$
\| \Delta^r u (t_{j+1}) \|^2_{L^2} - \| \Delta^r u(t_j) \|^2_{L^2} =
\int_{[t_j,t_{j+1}]} \partial_t \| \Delta^r u \|_{L^2} \ud t
$$
and
\begin{eqnarray*} 
\partial_t \| \Delta^r u \|_{L^2} 
&=& 2 \Re \int_M \partial_t \Delta^r u \Delta^r \overline{u} \ud x \\  
&=& 2 \Re( \int_M \Delta^r (-i|u|^2u+i\Delta u) \Delta^r \overline{u} \ud x) \\
&=& 2 \Im( \int_M \Delta^r (\overline{u} u^2) \Delta^r \overline{u} \ud x), \\
\end{eqnarray*}
since $\Im (\int_M \Delta (\Delta^r u) \Delta^r
\overline{u}\ud x)=0$. 

Computing $\Delta^r (\overline{u} u^2)$, there are three types of terms.
The first type is the one where all the derivatives from $\Delta$ act on $u$:
$$
\int_M |\Delta^r u|^2 |u|^2 \ud x,
$$
which is real and so makes no contribution.
The second type is  the one where all the derivatives act on $\overline{u}$:
$$
\int_M (\Delta^r \overline{u})^2 (u)^2 \ud x, 
$$
which is the most difficult to deal with. In all other terms, the derivatives do not act on the same factor.
They are of the form:
$$
\int_M \partial_x^{\alpha_1} \overline{u} \partial_x^{\alpha_2} u 
\partial_x^{\alpha_3} u \Delta^r \overline{u} \ud x,
$$
where the integers $\alpha_j$ are such that  their sums equal to $2r=s$ and 
at most one of them is zero.

We first deal with the third type of terms where no $\alpha_j$ is zero.
The estimates from Lemma 9 are crucial.

\begin{eqnarray*}
\int_{[t_j,t_{j+1}]} \int_M \partial_x^{\alpha_1} \overline{u}
\partial_x^{\alpha_2} u \partial_x^{\alpha_3} u \Delta^r \overline{u}
\ud x \ud t
&\lesssim& \|\Delta^r \overline{u} \|_{X^{-s_0 -\varepsilon,b}} 
\|\partial_x^{\alpha_1} \overline{u} \partial_x^{\alpha_2} u
\partial_x^{\alpha_3} u\|_{X^{s_0 +  \varepsilon,-b}} \\
&\lesssim& \|\Delta^r u\|_{X^{-s_0 -\varepsilon,b}}
\|\partial_x^{\alpha_1}u\|_{X^{s_0 + \varepsilon,b}}
\|\partial_x^{\alpha_2}u\|_{X^{s_0 + \varepsilon,b}}
\|\partial_x^{\alpha_3}u\|_{X^{s_0 + \varepsilon,b}} \\ 
&\lesssim& \| u \|_{X^{s-s_0 -\varepsilon,b}}
\|u\|_{X^{\alpha_1+s_0+\varepsilon}} \|u\|_{X^{\alpha_2+s_0+\varepsilon}} 
\|u\|_{X^{\alpha_3+s_0+\varepsilon}} 
\end{eqnarray*}

Writing $X^{s-s_0 -\varepsilon,b}$ as the interpolate between $(X^{s,b} , 1 - \frac{s_0+\varepsilon}{s-1})$ and $(X^{1,b} ,\frac{s_0+\varepsilon}{s-1})$, one has:
$$
\| u \|_{X^{s-s_0 -\varepsilon,b}} \lesssim \|u\|_{X^{s,b}}^{1 -
  \frac{s_0+\varepsilon}{s-1}} \|u\|_{X^{1,b}}^{\frac{s_0+\varepsilon}{s-1}}
\lesssim \|u(t_j)\|^{1 - \frac{s_0+\varepsilon}{s-1}}_{H^s}. 
$$
Since each $\alpha_j$ is at least $1$,
$X^{\alpha_j+s_0+\varepsilon,b}$ is the interpolate 
between
$(X^{s,b} ,\frac{s_0+\varepsilon+\alpha_j-1}{s-1})$
and $(X^{1,b} , \frac{s-\alpha_j-s_0-\varepsilon}{s-1})$:
$$
\|u\|_{X^{\alpha_j+s_0+\varepsilon}} \lesssim
\|u\|_{H^s}^{\frac{s_0+\varepsilon+\alpha_j-1}{s-1}}. 
$$
The contribution of these terms are at most
$|u\|_{H^s}^{1+2 \frac{s_0+\varepsilon}{s-1} + \frac{s-3}{s-1}}$ 
i.e., $|u\|_{H^s}^{2-2\frac{1-s_0-\varepsilon}{s-1}}.$
The same computation gives a contribution of
$|u\|_{H^s}^{2-\frac{1-s_0-\varepsilon}{s-1}}$ when one $\alpha_j$ is allowed to be zero.

Returning to the term:
$$
\int_{M\times [0,T]} (\Delta^r \bar u)^2 u^2 \ud x \ud t , 
$$
let us write $u_0=\Delta^r u=u_1$, $u_2=u=u_3$ and prove
\begin{equation}
J = \int_{M\times [0,T]} u_0 u_1 u_2 u_3 \ud x \ud t \lesssim
\|u_0\|_{X^{0,b}} \|u_1\|_{X^{-c,b}} \|u_2\|_{X^{1,b}} \|u_3\|_{X^{1,b}} \label{dagdag}
\end{equation}
for some $c$ to be determined. 

Let $P_{\lambda}$ be the orthogonal projection onto the eigenspace with 
eigenvalue $\lambda^2$ as before and decompose
$u_j= \sum_{N_j} u_j^{N_j} = \sum_{N_j} \sum_{N_j \leq \lambda \leq
  2N_j} P_\lambda u_j$  
with $N_j$ being dyadic integers. Then,
$$
|J| \leq \sum_N J(N) = \sum_N |\int_{M\times [0,T]} u_0^{N_0}
u_1^{N_1} u_2^{N_2} u_3^{N_3} \ud x \ud t|.
$$

Since terms such that $N_0 \geq N_1+N_2+N_3$ are taken care of by Lemma \ref{zoin}, 
we suppose that $N_0 \leq N_1+N_2+N_3$. 
Let $0< \delta<1 $, $\varepsilon$ and $b>1/2$. They will be chosen later, $\varepsilon$ is only meant to be as small as desired so $2 \varepsilon$ will be writen as $\varepsilon$ to avoid useless computation. 
Let us first study the terms where $N_1^\delta \leq N_2$, (the terms where  
$N_1^\delta \leq N_3$ are similar).
\begin{eqnarray*}
J(N) 
&\leq& \|u_0 u_2\|_{L^2(M \times [0,T])} \|u_1 u_3\|_{L^2(M \times
  [0,T])} \\
&\leq& \min(N_0,N_2)^{s_0+\varepsilon}  \min(N_1,N_3)^{s_0} \|u_0\|_{X^{0,b}}
\|u_1\|_{X^{0,b}} \|u_2\|_{X^{0,b}} \|u_3\|_{X^{0,b}} \\ 
&\leq& (N_0 N_1 N_2 N_3)^{-\varepsilon} (N_0 N_1)^\varepsilon (N_2
N_3)^{s_0-1+\varepsilon} \|u_0\|_{X^{0,b}} \|u_1\|_{X^{0,b}}
\|u_2\|_{X^{1,b}} \|u_3\|_{X^{1,b}}, \\ 
\end{eqnarray*}
where $s_0-1+\epsilon<0$ if $\varepsilon$ is small enough, $N_1^\delta \leq N_3$, $N_0 \leq N_1+N_2+N_3$ and of course, $N_1 \leq N_1+N_2+N_3$. So,
\begin{eqnarray*}
J(N) 
&\leq& (N_0 N_1 N_2 N_3)^{-\varepsilon} N_1^\varepsilon
N_3^{s_0-1+\varepsilon} \|u_0\|_{X^{0,b}} \|u_1\|_{X^{0,b}}
\|u_2\|_{X^{1,b}} \|u_3\|_{X^{1,b}} \\   
&\leq& (N_0 N_1 N_2 N_3)^{-\varepsilon} N_1^{\delta (s_0-1) + \varepsilon}
\|u_1\|_{X^{0,b}} \|u_2\|_{X^{1,b}} \|u_3\|_{X^{1,b}} \\   
&\leq& (N_0 N_1 N_2 N_3)^{-\varepsilon} \|u_0\|_{X^{0,b}}
\|u_1\|_{X^{\delta(s_0-1),b}} \|u_2\|_{X^{1,b}} \|u_3\|_{X^{1,b}}, \\  
\end{eqnarray*}
which proves (\ref{dagdag}) in this case with $c=\delta(s_0-1)$.

For the terms where $N_1^\delta \geq N_2$ et $N_1^\delta
\geq N_3$, more decomposition is needed. Let
$$
u_j^{N_j,L_j}= \frac{1}{2\pi} \sum_{N_j \leq \lambda \leq
  2N_j} \int_{\tau+\lambda}e^{it\tau} \widehat{P_\lambda{u}_j}(\tau)d\tau.  
$$
(We remark that this decomposition, the $X^{s, b}$ spaces and norms are inspired by the same
 intuition.)
Hence,
\begin{eqnarray*}
J(N) &\leq& \sum_L |\int_{\R \times M} u_0^{N_0,J_0} u_1^{N_1,J_1}
u_2^{N_2,J_2} u_3^{N_3,J_3} \ud x \ud t| \\
&\leq& \sum_L |\int_{\{\tau_0+\tau_1+\tau_2+\tau_3=0\} \times M}
\hat{u}_0^{N_0,L_0} \hat{u}_1^{N_1,L_1} \hat{u}_2^{N_2,L_2} \hat{u}_3^{N_3,L_3}
\ud x \ud \tau| \\ 
\end{eqnarray*}
For each term of this sum, we will decompose the integral depending on whether 
 $|\tau_2| \leq 1/4 N_1^2$ and $|\tau_3| \leq  1/4 N_1^2$ or if one of these statements fails.

In order to deal with the first case, notice that :
\begin{eqnarray*}
|\tau_0 + \lambda_0|+|\tau_1+ \lambda_1|
&\geq& |\tau_0+\tau_1+\lambda_0+\lambda_1| \\
&\geq& |\lambda_0+\lambda_1| - |\tau_0+\tau_1|\\
&\geq& \lambda_1 - |\tau_2 + \tau_3| \\
&\geq& 1/2 N_1^2 \\
\end{eqnarray*}

Hence, if $L_0+L_1 \geq 1/2 N_1^2$ does not hold, the term will be zero.
Then $L_0L_1 \geq 1/2 N_1^2-1$.
So the term can be bounded by :
\begin{eqnarray*}
J(N,L)  &\lesssim& \|u_0^{N_0,L_0}\|_{L^4_t L^2_x}
\|u_1^{N_1,L_1}\|_{L^4_t L^2_x} 
\|u_2^{N_2,L_2}\|_{L^2_t L^\infty_x} \|u_3^{N_3,L_3}\|_{L^2_t L^\infty_x} \\
&\lesssim& (N_2 N_3)^{d/2} \Pi_j \|u_j^{N_j,L_j}\|_{L^4_t L^2_x} \\ 
&\lesssim& (N_2 N_3)^{d/2} \Pi_j \|u_j^{N_j,L_j}\|_{X^{0,1/4}} \\
&\lesssim& \frac{(N_2 N_3)^{d/2-1}}{(L_0 L_1 L_2 L_3)^{b-1/4}}
\|u_0\|_{X^{0,b}} \|u_1\|_{X^{0,b}}  \|u_2\|_{X^{1,b}} \|u_3\|_{X^{1,b}}  
\end{eqnarray*}

Summing over $L$, we obtain
\begin{eqnarray*}
&\quad&\sum_L J(N,L)\\ 
&\lesssim& \sum_{L_0,L_1} \frac{(N_2 N_3)^{d/2-1}}{(L_0L_1)^{b-1/4}}
\|u_0\|_{X^{0,b}} \|u_1\|_{X^{0,b}} \|u_2\|_{X^{1,b}} \|u_3\|_{X^{1,b}} \\
&\lesssim& \sum_{L_0,L_1} \frac{N_1^{2(1/4-b+\varepsilon)} (N_2
  N_3)^{d/2-1}}{(L_0L_1)^{\varepsilon}}  \|u_0\|_{X^{0,b}}
\|u_1\|_{X^{0,b}} \|u_2\|_{X^{1,b}} \|u_3\|_{X^{1,b}} \\
&\lesssim& N_1^{-2(b-1/4-\varepsilon)} (N_2 N_3)^{d/2-1} \|u_0\|_{X^{0,b}}
\|u_1\|_{X^{0,b}} \|u_2\|_{X^{1,b}} \|u_3\|_{X^{1,b}} \\
&\lesssim& (N_0 N_1 N_2 N_3)^{-\varepsilon} N_0^\varepsilon
N_1^{-2(b-1/4)} N_2^{d/2-1+\varepsilon} N_3^{d/2-1+\varepsilon}
\|u_0\|_{X^{0,b}} \|u_1\|_{X^{0,b}} \|u_2\|_{X^{1,b}} \|u_3\|_{X^{1,b}} \\
&\lesssim& (N_0 N_1 N_2 N_3)^{-\varepsilon} N_1^{-2(b-1/4) +
  \varepsilon (2+2\delta) + \delta (d-2)} 
\|u_0\|_{X^{0,b}} \|u_1\|_{X^{0,b}} \|u_2\|_{X^{1,b}} \|u_3\|_{X^{1,b}} \\
&\lesssim&  (N_0 N_1 N_2 N_3)^{-\varepsilon} \|u_0\|_{X^{0,b}}
\|u_1\|_{X^{c,b}}  \|u_2\|_{X^{1,b}} \|u_3\|_{X^{1,b}}, \\
\end{eqnarray*}
where $c'={-2(b-1/4) +  \varepsilon (2+2\delta) + \delta (d-2)}$. 

Finally in the second case $|\tau_2| > 1/4 N_1²$ (or $|\tau_3| > 1/4 N_1²$).
Notice that
$$
|\tau_2 + \lambda_2| > \tau_2 - \lambda_2 > 1/2 N_1².
$$
Then unless $L_2 > N_1²$, the term is zero. 
So computing as previously :
$$
J(N,L) \lesssim 
\frac{(N_2 N_3)^{d/2-1}}{(L_0 L_1 L_2 L_3)^{b-1/4}}
\|u_0\|_{X^{0,b}} \|u_1\|_{X^{0,b}}  \|u_2\|_{X^{1,b}} \|u_3\|_{X^{1,b}}. 
$$

Summing over $L$:
\begin{eqnarray*}
&\quad&\sum_L J(N,L)\\
&\lesssim& \sum_{L_2} (N_2 N_3)^{d/2-1} L_2^{-(b-1/4)} \|u_0\|_{X^{0,b}}
\|u_1\|_{X^{0,b}}  \|u_2\|_{X^{1,b}} \|u_3\|_{X^{1,b}} \\
&\lesssim& \sum_{L_2} (N_0 N_1 N_2 N_3)^{-\varepsilon} N_0^\varepsilon
N_1^{\varepsilon - 2(b-1/4)} (N_2 N_3)^{d/2-1 + \varepsilon}
L_2^{-\varepsilon} \|u_0\|_{X^{0,b}} \|u_1\|_{X^{0,b}}  \|u_2\|_{X^{1,b}}
\|u_3\|_{X^{1,b}} \\
&\lesssim& (N_0 N_1 N_2 N_3)^{-\varepsilon} N_1^{-2(b-1/4) +\varepsilon +2\delta(d/2-2)}
\|u_0\|_{X^{0,b}} \|u_1\|_{X^{0,b}}  \|u_2\|_{X^{1,b}} \|u_3\|_{X^{1,b}} \\
&\lesssim& (N_0 N_1 N_2 N_3)^{-\varepsilon} \|u_0\|_{X^{0,b}}
\|u_1\|_{X^{c'',b}}  \|u_2\|_{X^{1,b}} \|u_3\|_{X^{1,b}}, \\ 
\end{eqnarray*}
where $c''= -2(b-1/4) +\varepsilon +2\delta(d/2-1)$.

Recall we have just proved that
$
\| \Delta^r u (t_{j+1}) \|^2_{L^2} - \| \Delta^r u(t_j) \|^2_{L^2} 
$
can be written as the sum of terms which are respectively bounded by
$\|u\|_{H^s}^{2-2\frac{1-s_0-\varepsilon}{s-1}}$, $\|u\|_{H^s}^{2-1\frac{1-s_0-\varepsilon}{s-1}}$, $\|u\|_{H^s} \|u\|_{H^{s-c}}$,$\|u\|_{H^s} \|u\|_{H^{s-c'}}$ or $\|u\|_{H^s} \|u\|_{H^{s-c}}$, since $\|u\|_{H^1}$ is bounded with $c= \delta(s_0-1)$, $c'= c''=  -2(b-1/4) +\varepsilon +2\delta(d/2-1) $. 
When $d=2$, $c>c'$ for any $\delta<1$. When $d\geq 3$, we
equate $c$ with $c'$ and let $\delta = \frac{2(b-1/4)}{d-s_0-1} $, then
$c=\frac{2(b-1/4)(s_0-1)}{d-s_0-1}$
By interpolation, one has
$\|u\|_{H^s} \|u\|_{H^{s-c'}} \leq |u\|_{H^s}^{2-\frac{c}{s-1}}$  
as $b$ can be choosen close to $1/2$, so $c=\frac{(1-s_0)}{2(d-s_0 )}$.

Using the above bounds and integrating the differential inequality, we obtain the proposition.
\end{proof}

\bibliography{references}

\end{document}